\documentclass[12pt, a4paper]{article}
\usepackage{latexsym}
\usepackage{amsmath,amssymb,amsfonts}

\newtheorem{prethm}{{\bf Theorem}}

\newenvironment{thm}{\begin{prethm}{\hspace{-0.5
               em}{\bf.}}}{\end{prethm}}

\newtheorem{prepro}[prethm]{Proposition}
\newenvironment{pro}{\begin{prepro}{\hspace{-0.5
               em}{\bf.}}}{\end{prepro}}

\newtheorem{prelem}[prethm]{Lemma}
\newenvironment{lem}{\begin{prelem}{\hspace{-0.5
               em}{\bf.}}}{\end{prelem}}

\newtheorem{precor}[prethm]{Corollary}

\newtheorem{preremark}{{\bf Remark}}

\newenvironment{rem}{\begin{preremark}\em{\hspace{-0.5
              em}{\bf.}}}{\end{preremark}}

\newtheorem{preexample}{{\bf Example}}

\newtheorem{preproblem}{{\bf problem}}

\newtheorem{preproof}{{\bf Proof.}}

\newenvironment{proof}[1]{\begin{preproof}{\rm
               #1}\hfill{$\Box$}}{\end{preproof}}

\renewcommand{\thefootnote}

\begin{document}
\title{  Finite unitary rings    with    a single subgroup   of prime  order of the group of units  }
\author{ Mostafa Amini, Mohsen  Amiri\\
{\footnotesize {\em Department of Mathematics, Payame Noor University, Tehran, Iran  }}\\
{\footnotesize{\em Departamento de Matem\'atica-ICE-UFAM, 69080-900, Manaus-AM, Brazil}}}
\footnotetext{E-mail Address: {\tt;amini.pnu1356@gmail.com\,  m.amiri77@gmail.com} }
\date{}
\maketitle

%------------------------------------------------------------------------------------------------------------------------------
\begin{quote}
{\small \hfill{\rule{13.3cm}{.1mm}\hskip2cm} \textbf{Abstract.}

Let  $R$  be  a  unitary ring of finite cardinality $p^{\beta}k$, where $p$ is a prime number and $p\nmid k$.   We show that if    the group of units of $R$ has at most one subgroup of order $p$, then    $R\cong A\bigoplus B,$ where $B$ is a finite ring of order $k$ and $A$ is a ring of   cardinality $p^{\beta}$   which is one of the six
explicitly described types. \\[.2cm]
%\vspace{1mm} {\renewcommand{\baselinestretch}{1}
%\parskip = 10 mm
}\\
\vspace{-3mm}\hfill{\rule{13.3cm}{.1mm}\hskip2cm}
\end{quote}
%_

\section{Introduction}

Finite rings are one of the topics of interest in ring theory and have been studied from many different points of view.
 Several mathematicians
study the finite unitary rings  through   their groups of units (the group of units of a ring is also called Adjoint Group). For example,  in 1968, Eldridge in \cite{1112}   proved that if  $R$ is a finite ring with unit of order $m$ such that  $m$ has cube free factorization, then $R$ is a commutative ring.
 In 1989, Groza in  \cite{22} showed that if $R$ is a finite ring and at most  one simple
component of the semi-simple ring $\frac{R}{J(R)}$ is a field of order 2, then $R^*$ (units of $R$) is a
nilpotent group if and only if $R$ is a direct sum of two-sided ideals that are
homomorphic images of group algebras of type $SP$, where $S$ is a particular
commutative finite ring and  $P$ is a finite $p$-group for  a prime number $p$.
In 2009,  Dolzan in \cite{kesa} improved  this result and described  the structure of
finite rings in which the group of units is  nilpotent. 
J. F. Waters in \cite{Wat} gave a ring-theoretical characterisation of the class of radical rings  with finitely generated nilpotent group of units. 
 J. B. Bateman and D. B. Coleman in \cite{Bat} determined under what conditions the group of units
in the group algebra of a finite group is nilpotent. 
P. B. Bhattacharya and S. K. Jain in \cite{Jain1}  studied  certain classes of rings R for which
$R^*$ is nilpotent, supersolvable or solvable. Also, in  \cite{Jain2} they 
proved that if $F$ is a field of characteristic $p$ and  $G$ a finite group,   then the group
of units of the group ring $R=FG$ is nilpotent if and only if $G$ is nilpotent and each
$q$-sylow subgroup of $G$, $q \neq p$, is abelian.
 Recently, in 2018, M. Amiri and M. Ariannejad   in \cite{mohsen} classified all finite unitary rings in which all Sylow subgroups of the group of units are cyclic.
Here, we improve this result and we  classify all finite unitary rings of order $p^{\alpha}k$, where $p$ is a prime number and $p\nmid k$ in which  $R^*$ has one subgroup of order $p$.
Precisely,   we prove the following theorem:\\[.2cm]

{\bf Theorem.}\label{thm2222}
 Let  $R$  be  a  unitary ring of finite cardinality $p^{\beta}k$, where $p$ is a prime number and $p\nmid k$.   If    $R^*$ has at most one subgroup of order $p$, then, $R=A\bigoplus B,$ where $B$ is a ring of order $k$ and $A$ is a ring of
cardinality $p^{\beta}$ which is isomorphic to  one of the following six  types:
 $$ {\mathbb{Z}}_{p^{\alpha}},   M_2(GF(p)),   T_2(GF(2)),{\mathbb{Z}}_{p^{\alpha}}\bigoplus_{i=1}^k (GF(p^{n_i})),  M_2(GF(p))\bigoplus_{i=1}^c (GF(p^{n_i})),$$
$$     T_2(GF(2))\bigoplus_{i=1}^s (GF(2^{n_i})), \ n_i, k, c, s, \alpha\in \mathbb{N}.$$
Moreover, if $p=2$, then $\alpha\leq 2$.\\[.5cm]
%--------------------------------------------------------------------
 
The following are some necessary  preliminary concepts and notations.\\[.5cm]
Let $R$ be a ring with identity $1\neq 0$. We denote   the Jacobson radical of $R$ by $J(R)$, and the cardinal of a given set $X$ by $|X|$. The set of unit elements of $R$ (or the group of units of $R$) is denoted by $R^*$. For a given prime number $p$, we denote   the set of all Sylow $p$-subgroups of $R^*$ by  $Syl_p(R^*)$. 
 The prime subring of $R$, i.e. the subring   generated by identity element of $R$, is denoted by  $R_0$.  $R_0[S]$ is  the  subring generated by $\{S\cup R_0\}$  where $S$ is a subset of $R$.
The ring of all $n\times n$ matrices over $R$ is denoted by $M_n(R)$. The ring of all $n\times n$  upper triangular matrices is denoted by $T_n(R)$, and  the ring of integers modulo  $m$ is denoted by ${\mathbb{Z}}_m$.   We denote the characteristic of $R$ by $Char(R)$, and   $GF(p^m)$  denotes the unique finite field of characteristic $p$ and of order $p^m$.
The order of a group $G$ is denoted by $|G|$, and the subgroup generated by $g$  is denoted by $\langle g \rangle$.
Also, the order of an element $g\in G$ denoted by $o(g)$ is the order of $\langle g \rangle$. 
 
%%%%%%%%%%%%%%%%%%%%%%%%%%%%%%%%%%%%%%%%%%%%%%%%%%%%%%%%%%%%%%%%%%%%%%%%%%%%%%%%%%%%%%%%%%%%%%%%%%%%%%%%
%%%%%%%%%%%%%%%%%%%%%%%%%%%%%%%%%%%%%%%%%%%%%%%%%%%%%%%%%%%%%%%%%%%%%%%%%%%%%%%%%%%%%%%%%%%%%%%%%%%%%%%%%%%%%%%%%

\section{Results}
From here onwards  we assume that   $R$ is a finite unitary  ring of order equal to a power of prime number $p$,  unless we specifically mention otherwise.
We start with  the following  elementary  lemma:
\begin{lem}\label{2a}
Let   $I$ be an ideal of $R$ such that $I\subseteq J(R)$. If   Sylow $p$-subgroups of $R^*$ are  cyclic, then  all Sylow $p$-subgroups of $(\frac{R}{I})^*$ are cyclic.  In addition, we have
$(\frac{R}{I})^*=\frac{R^*+I}{I}$.
\end{lem}
\begin{proof}
{Canonical epimorphism  $f: R^*\longrightarrow (\frac{R}{I})^*$  defined by  $f(a)=a+I$ shows that every Sylow $p$-subgroup of  $(\frac{R}{I})^*$ is  cyclic.
Clearly, $\frac{R^*+I}{I}\subseteq (\frac{R}{I})^*$.  For the other side,  let $x+I\in (\frac{R}{I})^*$.
Then, there exists $y+I\in (\frac{R}{I})^*$ such that $xy+I=1+I$, and consequently
$xy-1\in I$. Since $I\subseteq J(R)$, we deduce that $xy=xy-1+1\in R^*$. So $x\in R^*$ and hence $x+I\in \frac{R^*+I}{I}$.
 }
\end{proof}

%-----------------------------------------------------------------------------------------------------
\begin{rem}\label{ddd1}
Let $R=A\bigoplus B$ be a  finite ring, where $A$ and $B$ are two ideals of $R$. Then, $R^*=A^*\bigoplus B^*$ and $1=1_A+1_B$, where $1_A$ and $1_B$ are identity elements of $A$ and $B$, respectively. It is clear that $A^*+1_B\leq R^*$ and $A^*+1_B\cong A^*$. In addition, if $p\mid gcd(|A^*|,|B^*|)$ for some prime number $p$, then by Cauchy Theorem, $R^*$ has two elements $a+1_B$ and $1_A+b$ with the same order $p$. Clearly, $\langle a+1_B\rangle\neq \langle 1_A+b\rangle$ and this impleis that Sylow $p-$subgroups of $R^*$ are not cyclic.
\end{rem}
We also need the following lemma,  which is a direct consequence of   Lemma 1.1 of \cite{22}.
\begin{lem} \label{www}Let  $R$  be  a finite unitary ring of odd cardinality. Then, $R = R_0[R^*]$.
\end{lem}
%-------------------------------------------------------------------------------------------------
In the first step, we characterize all  unitary  finite rings $R$ with $R=R_0[R^*]$ in which all Sylow $p-$subgroups are cyclic.
\begin{pro}\label{thm222} Suppose that $R=R_0[R^*]$.
   If all Sylow $p-$subgroups of  $R^*$ are cyclic, then $R$ is a commutative ring or $R\cong M_2(GF(p))$ or  $R\cong M_2(GF(p))\bigoplus B,$ where  $B$ is a direct sum of finite fields of characteristic $p$.
\end{pro}
\begin{proof}{We proceed by induction on $|R|=p^{\beta}$. If $\beta<3$, then    by Lemma of \cite{1112} (page 512), $R$
 is commutative.  So   suppose that $\beta \geq 3$.
 We may assume that $R$ is not commutative.    If there exists  proper unitary subring $S$ of $R$ such that $S=S_0[S^*]$, then by induction hypothesis, $S$ is a commutative ring or $S\cong M_2(GF(p))$ or  $S\cong M_2(GF(p))\bigoplus B,$ where  $B$ is a direct sum of finite fields of characteristic $p$.
 We have two cases with respect to the  Jacobson radical, either  $J(R)=0$ or $J(R)\neq 0$:\\[.5cm]
%---------------------------------------------------------------------------------------------------
{\bf Case 1.} If $J(R)=0$, then $R$ is a simple artinian ring, and so by the structure theorem of Artin-Wedderburn,  we have $R\cong \bigoplus_{i=1}^t M_{n_i}(D_i),$ where  any $D_i$ is a finite field.   If $n_i>2$ for some $1\leq i\leq t$, then Sylow $p-$subgroup of $R$ has order bigger than $p$. Therefore, it is not cyclic,  which is a  contradiction. Since $R$ is not commutative, we may assume that  $n_1=2$. If $n_i>1$ for some $2\leq i\leq t$, then  Sylow $p-$subgroup of $R$ has order bigger than $p$. Therefore, it is not cyclic,   which is a  contradiction. Consequently, $R\cong M_{n_1}(GF(p))\bigoplus B,$ where  $B$ is a direct sum of finite fields of characteristic $p$.    
  \\[.5cm]
 %---------------- ---------------------------------------------------------------
{\bf Case 2.}  Suppose that $J(R)\neq 0$.
Let $I\subseteq J(R)$ be a minimal ideal of $R$ and let  $char(I)=p^i$. If $i>1$, then
      $Ip$ is a non-trivial ideal of $R$, and so
      $I=Ip$. Let $s\in I$. Then, $s=\sum vp$ for some $v\in I$.
    It follows that $sp^{i-1}=\sum vp^i=0$, and hence  $char(I)=p^{i-1}$,  a contradiction. Hence, $char(I)=p$. Clearly, $I^2=0$. Thus,
  for all $s\in I$, we have $(1+s)^p=1$.  Therefore   $1+I$ is an elementary abelian $p-$group. Since Sylow $p-$subgroups of $R^*$ are cyclic, we have $|1+I|=|I|=p$. Hence, $I=\{0,a,2a,3a,\ldots,(p-1)a\}$ for any non-zero element $a\in I$.  On the other hand, $|ann_R(a)|=\frac{|R|}{p}$.
  Then, for each $x\in R$ there exists an integer $1\leq i\leq p$ such that
  $x+i\in ann_R(a)$. Let $P\in Syl_p(R^*)$, and let $P=\langle z\rangle$.
   Since $1+I\lhd R^*$, by Normalizer-Centralizer Theorem,
there exists a  monomorphism $f$ from $\frac{P}{C_P(1+a)}$ into $Aut(1+I)\cong C_{p-1}$, and consequently, $1+I\leq Z(P)$.

Let $x\in R^*$ such that $gcd(o(x),p)=1$. Since $x^{-1}(1+a)x\in I$, there exists $1\leq j\leq p-1$ such that $x^{-1}(1+a)x=j(1+a)$. So,   $x^{-(p-1)}(1+a)x^{p-1}=j^{p-1}(1+a)=1+a$, and hence
$(1+a)\in C_{R^*}(x^p)=C_{R^*}(x^p)$. Thus, $1+a\in Z(R^*)$, and so $a\in Z(R)$.
Consequently, $1+I\leq Z(R^*)$ and  $I\subseteq Z(R)$, because $R=R_0[R^*]$.

First suppose that $1+I\neq P$.
Since $z^ja\neq 0$ for all integers $j$ and $I=\{0,a,\ldots,(p-1)a\}$, there exists an integer $2\leq i\leq p-1$ such that 
$ z^{i}a=a$.  Then, $(z^i-1)a=0$, and so $z^i-1\in ann_R(a)$.
Since $a\in Z(R)$, we deduce that $ann_R(a)$ is an ideal of $R$.
 Also, there exists an integer $1\leq j\leq p-1$ such that
$z-j\in ann_R(a)$. It follows that
$j^i-1\in ann_R(a)$, and so $j^i-1= 0  \ (mod \ p)$.
Consequently, $p\mid i$, which is a contradiction. Therefore,
$1+I=P$. Since  $1+J(R)\leq P$, we have
$J(R)=I$ and $1+J(R)=P$. By Wedderburn Structure Theorem,   $\frac{R}{J(R)}\cong \bigoplus_{i=1}^t M_{n_i}(D_i),$ where  any $D_i$ is a finite field.
 By induction hypothesis,  $\frac{R}{J(R)}$ is commutative or $\frac{R}{J(R)}\cong M_2(GF(p))$ or   $\frac{R}{J(R)}\cong M_2(GF(p))\bigoplus B,$ where  $B$ is a direct sum of finite fields of characteristic $p$.

First suppose that $\frac{R}{J(R)}\cong M_2(GF(p))$ or   $\frac{R}{J(R)}\cong M_2(GF(p))\bigoplus B,$ where  $B$ is a direct some of finite fields of characteristic $p$. Let $A$ be an ideal of $R$ such that $\frac{R}{A}\cong M_2(GF(p))$. Then, $|R|=|A|p^4$.
      Let $z\in (\frac{R}{A})^*$ with $o(z+A)>1$.
       We have $az\in J(R)=\{0,a,\ldots,(p-1)a\}$. Therefore, $az=ja,$ where $1\leq j\leq p-1$, and so   $z-j\in ann_R(J(R))$. Since $\frac{R}{A}$ is a simple ring and $ann_R(a)$ is a two sided ideal ($a \in Z(R)$), we deduce that $ann_R(a)\subseteq A$.
      On the other hand, we have $|ann_R(a)|=\frac{|R|}{p}$, and so we obtain that  $|R|\leq p|A|$, a contradiction. So, $\frac{R}{J(R)}$ is commutative.

Let $x, y\in R^*$ such that $p\nmid o(x)$ and $p\nmid o(y)$.
Since $xy-yx\in J(R)$, we have $[x,y]-1\in J(R)$. Consequently, $[x,y]\in Z(R^*)$, and
 $R^*$ is a nilpotent group. It follows that $R^{*}$ is a direct sum of  Sylow subgroups.
Let $x, y\in Q\in Syl_q(R^*)$, where $p\nmid q$.
Also, let $xy-yx=ja$ and  $xy=cyx,$ where $c\in Z(R^*)$.
Then, we have $cyx-yx=ja$. Hence
$(c-1)yx=ja$, and so $jax^{-1}y^{-1}=c-1\in \{1,2,\ldots,p\}$. Thus, $o(c)\mid p$, and then
$x^py=c^pyx^p=yx^p$. Since
$gcd(p,o(x))=1$, we have $xy=yx$.
Consequently, $R^*$ is an abelian group, and hence $R$ is commutative, which is  our final contradiction.
}
\end{proof}
Let $p$ be a prime number.
For simplicity,   let $\Gamma(p)$  be the set of all finite rings $A$ of the following types:

(a) If $p>2$, then    $A$ is a direct sum of finite fields of order $p$ or a direct sum of finite fields of order $p$ and  ${\mathbb{Z}}_{p^m}$ or a direct sum of finite fields of order $p$ and $M_2(GF(p))$ or $M_2(GF(p))$.

(b) If $p=2$, then    $A$ is a direct sum of finite fields of order $2$ or a direct sum of finite fields of order $2$ with  ${\mathbb{Z}}_{4}$ or a direct sum of finite fields of order $2$ with $M_2(GF(2))$ or $M_2(GF(2))$.
\begin{pro}\label{thm3} Let $R=R_0[R^*]$.
  If  all Sylow $p$-subgroups of  $R^*$ are cyclic, then $R\in \Gamma(p)$.
\end{pro}
\begin{proof}
{By Proposition \ref{thm222},  we may assume that $R$  is a commutative ring. Let $|R|=p^{\beta}$. Then,
   we proceed by
induction on $\beta$. First suppose that $|R|=p^2$. If $R$ is not a field, then   we have two possible isomorphism classes ${\mathbb{Z}}_{p^2}$  and ${\mathbb{Z}}_p\bigoplus {\mathbb{Z}}_p$. Clearly, $R\in \Gamma$ in this case.  Hence
 we may assume  that   $\beta>2$.  We consider two cases depending on the Jacobson radical: $J(R)=0$ or $J(R)\neq 0$.

 {\bf Case 1.} First, let $J(R)=0$.  Since  $R$ is a semi-simple ring,   by Wedderburn Structure Theorem, $R\cong \bigoplus_{i=1}^k R_i$   is a direct product of matrix algebras over division algebras.  $R$ is a commutative ring, so $R_i$ is a finite field for each $i$. On the other hand, $|R_i|=p^{\alpha_i}$ for some integer $\alpha_i$($\alpha_i\leq t$), and hence we have $|R_i^*|\geq p-1$.  Since the Sylow $p-$subgroup is cyclic,
 $|R_i|>p$ at most for one $i$, and so $R\in \Gamma(p)$.

   {\bf Case 2.}
 Let $J(R)\neq 0$, and
   $I\subseteq J(R)$ be a minimal ideal of $R$. Similar to the proof of Proposition \ref{thm222} case 2, we can prove that, $char(I)=p$, $I^2=0$ and   $1+I$ is an elementary abelian $p-$group. Since Sylow $p-$subgroups of $R^*$ are cyclic, we have $|1+I|=|I|=p$. So, $I=\{0,a,2a,3a,\ldots,(p-1)a\}$ for any non-zero element $a$ belonging  to $I$.
  Let $P\in Syl_p(R^*)$ and $P=\langle z\rangle$.
Since $z^ia\neq 0$ for all integers $i$ and $|I|=p$, we have
$ z^p=1$, and it follows that $|P|=p$. Also, $1+J(R)\leq P$, hence
$J(R)=I$ and $1+J(R)=P$.

Let $\{M_1,\ldots,M_e\}$ be the set of all maximal ideals of $R$. By the structure theorem of Artin-Wedderburn, there exists  an isomorphism $f$ from   $ \bigoplus_{i=1}^e \frac{R}{M_i}$ to $\frac{R}{J(R)}$.  For each $x\in R$, there exists $1\leq j\leq p$ such that $xa=ja$, and hence
$x-j\in ann_R(a)$. Thus, $\frac{|R|}{|ann_R(a)|}=p$, and so   $ann_R(a)$ is a maximal ideal of $R$.
 We may assume that $M_1=ann_R(a)$.

 Let $f((1+M_1,M_2,M_3,\ldots,M_k))= x+J(R)$.
 It is clear that $|\frac{ann_R(x)+J(R)}{J(R)}|=\frac{|R|}{p|J(R)|}$, and so $|Rx|=|\frac{R}{ann_R(x)}|=p|J(R)|$. Since $ax\neq 0$, we have $a\not\in ann_R(x)$.  Also, $Jac(R)$ is a minimal ideal of $R$ and $a\not\in ann _R(x)$, so $ann_R(x)\cap J(R)=0$.
Since $|R|=|ann_R(x)|p$, we deduce that
  $R=ann_R(x)\bigoplus Rx$.
   By induction hypothesis, $ann_R( x)$ and $ Rx$   belong to the set  $ \Gamma(p)$. Evidently, $R\in \Gamma(p)$, because $gcd((|ann_R( x))^*|,p)=gcd(|(Rx)^*|,p)=1$.

Hence we may assume that $e=1$. Then, $M_1=J(R)$ and $\frac{R}{J(R)}$ is a finite field.
It is easy to see that every Sylow $q-$subgroup of $R^*$ is cyclic for all $q\mid |R^*|$, and
consequently by main result of \cite{mohsen}, $R\in \Gamma(p)$.
}
\end{proof}

Now, by Lemma \ref{www}, it   remains  to characterize the ring of order $2^{\beta}$ in which all its  Sylow $2$-subgroups are cyclic and $R\neq R_0[R^*]$.

\begin{pro}\label{t2}
Let  $R$  be  a  unitary ring of finite cardinality $2^{\beta}$ and  $H=R_0[R^*]$ such that every Sylow $2$-subgroup of  $R^*$ is a cyclic group.

(a)
If $R$ is a commutative ring, then   $R\in \Gamma(2)$.

(b) If $H$ is a commutative ring and $R$ is a non-commutative ring, then either   $R\cong T_2(GF(2))$ or  $R\cong T_2(GF(2))\bigoplus_{i=1}^k (GF(2^{n_i}))$ for some positive integer $k$.
\end{pro}
\begin{proof}{We proceed by induction  $|R|$.  Let $I\subseteq J(R)$ be a minimal ideal of $R$.
Since Sylow $2-$subgroups of $R^*$ are cyclic, we have $|1+I|=|I|=2$. Therefore, $I=\{0,a\}$ for any non-zero element $a$ belongs to $I$.
Let $P\in Syl_p(R^*)$, and let $P=\langle z\rangle$. Then,
since $z^ia\neq 0$ for all integers $i$ and $|I|=2$, we have
$ z^2=1$ and it follows that $|P|=2$. But since $1+J(R)\leq P$, we deduce that
$J(R)=I$ and $1+J(R)=P$.  It is clear that $|ann_R(a)|=\frac{|R|}{2}$ and
by Proposition of \cite{1}, every unitary  non-commutative ring of order $8$ is isomorphic to $T_2(GF(2))$. So  we may assume that $|R|>8$.
Also, $J(\frac{R}{I})=0$. Thus
by Wedderburn Structure Theorem, $\frac{R}{I}\cong \bigoplus_{i=1}^k R_i$   is a direct product of matrix algebras over division algebras, and since $R_0[R^*]$ is commutative,   $R_i$ is a finite field for each $i=1,2,\ldots,k$.

      Let $\{M_1,\ldots,M_k\}$ be the set of all maximal ideals of $R$.
 If $k=1$, then  $J(R)=M_1=ann_R(a)$, because $\frac{R}{I}$ is commutative.
Since $[R:ann_R(a)]=2$, we have $R=R_0[(1+J(R))]=R_0[R^*]$, which is a contradiction. So
 $k>1$.  We may assume that $M_1=ann_R(a)$.
 
Let $f$ be an ismorphism from  $ \frac{R}{M_1}\bigoplus\frac{R}{M_2}\bigoplus\ldots\bigoplus\frac{R}{M_k}$  to $ \frac{R}{J(R)}$.
 Let \break $f((1+M_1,M_2,M_3,\ldots,M_k))= x+J(R)$.  Since $I$ is the unique minimal ideal of $R$ and $a\not\in ann _R(x)$, we have    $ann_R(x)\cap J(R)=0$.
 It is clear that $|\frac{ann_R(x)+J(R)}{J(R)}|=\frac{|R|}{2|J(R)|}$. Consequently, $|Rx|=|\frac{R}{ann_R(x)}|=2|J(R)|$. Since $ax\neq 0$, we have $a\not\in ann_R(x)$.  Also, $I=J(R)$. So
 $I\subseteq Rx$,  and hence $R=ann_R(x)\bigoplus Rx$.
 
Clearly,   $x$ is the identity element of $Rx$. Hence
 $(x+a)^2=x^2=x$, and so  $x+a$ is an element of $Rx$ of order $2$.
Let $y\in ann_R(x)$ be the identity element of   $ann_R(x)$ and
let  $u\in ann_R(x)$ such that   $u^2=y$.    Since  $o(u+(x+a))=o(y+(x+a))=2$ and the Sylow $2-$subgroup is cyclic, we have
$u+x=y+x+a$, and hence $u-y=a\in ann_R(x)$, a contradiction. Therefore, $|(ann_R(x))^*|$ is an odd number.  Since $J(ann_R(x))=0$, we have
$ann_R(x)\in \Gamma(2)$.

If $R$ is a commutative ring, then $R\cong {\mathbb{Z}}_{4}\bigoplus ann_R(x)\in \Gamma(2)$.
If $R$ is not a commutative ring, then  $Rx$ is not a commutative ring, because  otherwise $x(ax)=(ax)x=ax^2=x$, and so $(xa-1)x=0$, consequently, $x=0$, which is a contradiction. So by induction hypothesis,  $Rx\cong T_2(GF(2))$ or $R\cong T_2(GF(2))\bigoplus_{i=1}^m (GF(2^{n_i}))$ for some positive integer $m$.
  }
\end{proof}

\begin{thm}\label{thm2121}
 Let  $R$  be  a  unitary ring of finite cardinality $p^{\beta}$, where $p$ is a prime.   If  all Sylow $p$-subgroups of  $R^*$ are cyclic, then,
$R$  is isomorphic to  one of the following six  types:
 $$  {\mathbb{Z}}_{p^{\alpha}},   M_2(GF(p)),   T_2(GF(2)),{\mathbb{Z}}_{p^{\alpha}}\bigoplus_{i=1}^k (GF(p^{n_i})),  M_2(GF(p))\bigoplus_{i=1}^c (GF(p^{n_i})),$$
$$     T_2(GF(2))\bigoplus_{i=1}^s (GF(2^{n_i})), \ n_i,k,c,s,\alpha\in \mathbb{N}.$$
Moreover, if $p=2$, then $\alpha\leq 2$.
 
 \end{thm}

%------------------------------------------
 %-----------------------------------------------------------------------------------------------------------------------------------------------
\begin{proof}{ The proof is clear by Propositions \ref{thm222}  and  \ref{t2}.
}

\end{proof}
Every element of order two in any group is called an involution.
The following lemma helps us to classify all finite unitary rings with a single  involution.

\begin{lem}\label{t22}
Let  $R$  be  a  unitary ring of finite cardinality $2^{\beta}$ and  $H=R_0[R^*]$ such that  $R^*$ has only one involution. Then, all Sylow $2-$subgroups of $R^*$ are cyclic.

\end{lem}
\begin{proof}{We proceed by induction on $|R|$. Let $P\in Syl_2(R^*)$. The $2$-groups with unique involution were determined by Burnside; they must be cyclic or generalized quaternion groups. Hence, if $R$ is commutative,  then Sylow $2$-subgroup is cyclic. So suppose that $R$ is a non-commutative ring.   Then,  $$P=Q_{4n}=\langle \{u,l: u^{2n}=l^4=1, u^n=l^2, l^{-1}ul=u^{-1}\} \rangle .$$

The case  $|R|=2^3$,  is clear by  Proposition of \cite{1}. Assume that $|R|>2^3$.

First, let $J(R)=0$. Since  $R$ is a semi-simple ring,   by Wedderburn Structure Theorem, $R\cong \bigoplus_{i=1}^k R_i$   is a direct product of matrix algebras over division algebras.   On the other hand,  $R$ is not commutative. So, we  may assume that $n_1>1$, and in this case, the number of involutions of $R$ is more than one, which is a contradiction. Hence $J(R)\neq 0$.
Since $1+J(R)$ is a $2$-group, we have $|J(R)|\leq 4n$. Now, let $I\subseteq J(R)$ be a minimal ideal of $R$. It is easy to prove that $1+I$ is an elementary abelian $2-$group, and since  $R^*$ has only one involution, we have $|1+I|=|I|=2$. Hence, $I=\{0,a \}$ for any non-zero element $a$ belonging to $I$.  Clearly, $|ann_R(a)|=\frac{|R|}{2}$ and $1+a\in Z(R^*)$.
If $|P|\geq 16$, then $\frac{R}{I}$ has only one involution, because $\frac{P+I}{I}\cong Q_{2{(n-1)}}$, and hence by induction hypothesis,
$\frac{P+I}{I}$ is cyclic, which is a contradiction.  So $|P|\leq 8$.

Let $\{M_1,\ldots,M_k\}$ be the set of all maximal ideals of $R$.
 If $k=1$, then   $J(R)=M_1$.
 If $|\frac{R}{J(R)}|>2$, then by Lemma 1.1 of [10] we have  $R=R_0[R^*]$, and so
 $a\in Z(R)$. It follows that $ann_R(a)$ is a two sided ideal of $R$, and so $ann_R(a)=Jac(R)$. But $|\frac{R}{ann_R(a)}|=2$, which is a contradiction. So $\frac{R}{J(R)}$ is a finite field. 
 Let $w$ be a generator for the cyclic group $(\frac{R}{J(R)})^*$.
 Since $w-1\in ann_R(a)$, we have $|\frac{R}{J(R)}|=2$, and so
 $J(R)=ann_R(a)$. Hence $|R|\leq 16$. If $|J(R)|\leq 4$, then $|R|\leq 8$, which is a contradiction. So $|J(R)|=8$, and then $|J(\frac{R}{I})|=4$.
 If $\frac{R}{I}$ is not commutative, then by Proposition of \cite{1}, every unitary  non-commutative ring of order $8$ is isomorphic to  $T_2(GF(2))$.
 But $|J(T_2(GF(2))|=2$, which is a contradiction. Therefore $\frac{R}{I}$ is   commutative.
 Since $\frac{G}{J(R)}\cong GF(2)$, we have $u+l, u^3+l\in J(R)\setminus Z(R)$, and consequetly $u+l+1$ and $u^3+l+1$  have   order 4.
 We have $P=\langle u\rangle\cup \langle l\rangle \cup \langle ul\rangle$.
 It follows from $P$ is not an abelian group that, $u+l+1, u^3+l+1\in  \langle ul\rangle$. Thus,
 $u+l+1=ul$ or $u^3+l+1=ul$. If
 $u+l+1=ul$, then $u(l-1)=l+1$, and so
 $lu(l-1)=u(l-1)l$. Hence, we have $lu=ul$, which is a contradiction.
 If
 $u^3+l+1=ul$, then $(u-1)l=u^3+1$, and hence
 $u(u-1)l=(u-1)lu$.Thus
 $u^2l=ulu$, and so $lu=ul$, which is a contradiction. So  $k>1$.

 We may assume that $M_1=ann_R(a)$.
Let  $f: \frac{R}{M_1}\bigoplus\frac{R}{M_2}\bigoplus\ldots\bigoplus\frac{R}{M_k}\cong \frac{R}{J(R)}$ and let  $f(1+M_1,M_2,M_3,\ldots,M_k)= x+J(R)$. Then  $ax\neq 0$ imply that  $a\not\in ann_R(x)$.   Since $I$ is the unique minimal ideal of $R$, we conclude that $ann_R(x)\cap J(R)=0$.
 But $|\frac{ann_R(x)+J(R)}{J(R)}|=\frac{|R|}{2|J(R)|}$ and consequently, $|RxR|=|\frac{R}{ann_R(x)}|=2|J(R)|$. So $ax\in RxR$ implies that
 $I\subseteq RxR$. Since $ann_R(x)\cap RxR$ is a two sided ideal of $R$, we have $ann_R(x)\cap RxR=0$ or
 $I\subseteq ann_R(x)\cap RxR$. The later case is impossible, so $ann_R(x)\cap RxR=0$ and then $R=ann_R(x)\bigoplus RxR$. Hence  $ann_R(x)$ and $RxR$ are unitary rings.
First assume that
 $|(ann_R(x))^*|$ is not an odd number.
If  $h$ is  the identity element of $RxR$, then
 $(h+a)^2=h^2=h$, and so  $h+a$ is an element of $RxR$ of order $2$.
Let $y\in ann_R(x)$ be the identity element of   $ann_R(x)$ and
let $u\in ann_R(x)$ such that   $u^2=y$.    Since  there is only one involution and  $o(u+(h+a))=o(y+(h+a))=2$, we see that
$u+h=y+h+a$. Consequently,  $u-y=a\in ann_R(x)$, which is  a contradiction. Therefore,  $|(ann_R(x))^*|$ is an odd number, and then by induction hypothesis, Sylow $2$-subgroup of $(RxR)^*$ is cyclic, which is a contradiction.
}
\end{proof}
Now, we can prove our main theorem.\\[.5cm]

{\bf Proof of Theorem \ref{thm2222}}. Let  $|R|=p_1^{\alpha_1}\ldots p_k^{\alpha_k}$ be the canonical decomposition of $|R|$ into its  prime divisors $p_i$. Then $R = R_1\bigoplus R_2\bigoplus \ldots\bigoplus R_k,$ where each $R_i$ is an ideal of order $p_i^{\alpha_i}$.
Hence we may assume that $|R|=p^m$. Let $P\in Syl_p(R^*)$. If $p>2$, then $P$ is cyclic and the proof is clear by Theorem \ref{thm2121}.
If $p=2$, then $P$ is cyclic or $P$ is generalized quaternion groups.
The latter case is impossible by Lemma \ref{t22}. Hence $P$ is cyclic, and so we have the result by Theorem  \ref{thm2121}.


\begin{thebibliography}{99}

\bibitem{mohsen} M. Amiri,  M. Ariannejad, {Finite unitary rings in which all sylow
subgroups of the group of units are cyclic}, {\textit{Bull. Aust. Math. Soc}}.    Vol 99, Issue 3, (2019) 403-412.
\bibitem{Bat}J. B. Bateman and D. B. Coleman, Group Algebras with Nilpotent Unit Groups, {\textit{Proc. Math. Soc}}.1 9(1968),448-449.


\bibitem{Jain1}P.B.Bhattacharya and S. K. Jain, Rings with Solvable Adjoint Groups, {\textit{ Proc. Amer. Math. Soc}}. (1970), 563-565
\bibitem{Jain2}P.B.Bhattacharya and S. K. Jain, A Note on the Adjoint Group of a Ring, {\textit{Archiv der Math.}} Basel, (1070), 366-368


\bibitem{1} D.B. Erickson, Orders for finite noncommutative rings, {\textit{ The American Mathematical Monthly}}, 73 (1966) 376-377.
\bibitem{kesa} D. Dolzan, {Nilpotency of the group of units of a finite ring}, {\textit{Bull. Aust. Math. Soc}}, 79 (2009) 177-182.

\bibitem{1112}
K.E. Eldridge,
{Orders for finite noncommutative rings with unity}, {\textit{The American Mathematical Monthly}}, 75,  5 (1968) 512-514.

\bibitem{11122}
 B. Farb, R. Keith Dennis, {Noncommutative algebra},{\textit{ Graduate texts in mathematics, Springer Verlag, New York}}, 144, 1993.
\bibitem{12}
 B. Fine, {\textit{Classification of Finite Rings of Order $p^2$}}, Fairfield University, Fairfield, CT 06430.

\bibitem{22}
G. Groza,
{Artinian rings having a nilpotent group of units}, {\textit{J. Algebra}}, 121 (1989) 253-262.

\bibitem{5} T.Y. Lam, {\textit{A first course in  noncommutative rings}}, Gratuate texts in mathematics,   Springer Verlag, New York, 2nd ed., 2001.
\bibitem{Rob}D.J.S. Robinson, {\textit{A course in the theory of groups}}, Springer Verlag, New York, 1980.
\bibitem{Wat}
J. F. Waters, On the Adjoint Group of a Radical Ring, {\textit{J. London Math .Soc}}. (1968),725-729.
\bibitem{Wedderburn} J.H.M. {Wedderburn, A theorem  on finite algebra}, {\textit{Trans. Amer. Math. Soc}}., {\bf 6}(1905)\,349-352.
1996.
%---------------------------------------------------

\end{thebibliography}
\end{document}